\documentclass[11pt,a4paper]{article}
\usepackage{amsmath, amscd}
\usepackage{amssymb, amsthm}
\usepackage{mathrsfs, calrsfs}
\usepackage{graphicx}
\usepackage{textcomp}
\usepackage{array}
\usepackage{setspace}
\usepackage{enumitem}
\usepackage{makeidx}

\newtheorem{theorem}{Theorem}[section]
\newtheorem{remark}{Remark}[section]

\newtheorem{lemma}{Lemma}[section]

\newtheorem{definition}{Definition}[section]
\newtheorem{corollary}{Corollary}[section]

\title{Global existence and decay estimates for a viscoelastic plate equation with nonlinear damping and logarithmic nonlinearity}
\date{}
\author{Bhargav Kumar Kakumani \footnote{bhargav@hyderabad.bits-pilani.ac.in} \footnote{Department of Mathematics, BITS-Pilani, Hyderabad Campus, Hyderabad, India.}  , Suman Prabha Yadav \footnote{p20200454@hyderabad.bits-pilani.ac.in} \footnote{Department of Mathematics, BITS-Pilani, Hyderabad Campus, Hyderabad, India.}  
}

\begin{document}
	
	\maketitle
	\pagestyle{plain}
	\pagenumbering{arabic}

	\begin{abstract}
	In this article, we consider a viscoelastic plate equation with a logarithmic nonlinearity in the presence of nonlinear frictional damping term. Using the the Faedo-Galerkin method we establish the global existence of the solution of the problem and we also prove few general decay rate results.
	\end{abstract}
	

\section{Introduction}
This work deals with the existence and decay of solutions to the following plate problem:

\begin{equation}
\label{eqn_s1}
\left\{ \begin{array}{ll}
\left |u_{t} \right |^{\rho }u_{tt}+\Delta ^{2}u+\Delta ^{2}u_{tt}+u-\displaystyle\int_{0}^{t}b(t-s)\Delta ^{2}u(s)ds \medskip \\
\hspace{4cm} +h(u_{t})=ku \ln \left | u \right |, \hspace{0.5cm}  (x,t) \in \Omega \times(0,\infty), \medskip \\
u\left ( x,t \right )=\frac{\partial u}{\partial \nu}\left ( x,t \right )=0, \hspace{0.5cm} \text{in} \hspace{0.2cm} \, \partial \Omega \times(0,\infty),  \medskip\\
u(x,0)=u_{0}(x), \hspace{0.5cm}  u_{t}(x,0)=u_{1}(x) \hspace{0.5cm} \textrm{in} \hspace{0.2cm}\, \Omega
\end {array} \right.
\end{equation}
where $\Omega \subset \mathbb{R}^2$ is a bounded domain, $\nu$ is the unit outer normal to $\partial \Omega$, $k$ and $\rho $ are positive constants and $0 < \rho \leq \frac{2}{n-2}$ if $n \geq 3 $. We use the Lebesgue space $L^2(\Omega)$ and $H_{0}^{2}(\Omega)$ with their usual scalar product and norms. Through out this paper, we consider the following hypotheses:\\

\begin{enumerate}[label=(H\arabic*)]
	\label{eqn_s2}
	\item Let $b:\mathbb{R}^{+}\rightarrow \mathbb{R}^{+}$
	is a $C^1$-nonincreasing function satisfying
	\begin{equation}
	b(0)>0,\, 1-\int_{0}^{t}b(\tau)d\tau=l>0.
	\end{equation} \medskip 
	
	\label{eqn_s3}
	\item Assume that there exist a nonincreasing positive differentiable function $\xi$ such that   $\xi(0)>0$ and $\xi:\mathbb{R}^{+}\rightarrow \mathbb{R}^{+}$. Further assume that there exists a $C^1$ function $B:(0,\infty )\rightarrow (0,\infty )$ which is linear or strictly convex $C^2$ function and strictly increasing on $(0,r_1],\ r_1\leq b(0)$, $B(0)=B'(0)=0 $ and $B$ satisfies 
	\begin{equation}
	b'(t)\leq -\xi(t)B(b(t)), \quad \forall t\geq 0.
	\end{equation} \medskip 
	
	\item Let $h :\mathbb{R} \rightarrow \mathbb{R}$ is a nondecreasing continuous function such that there exist a $h_{1}\in C^{1}(\mathbb{R}^{+})$ with $ h_{1}(0)=0 $ which is strictly increasing function and $h_1$ satisfies
	
	\begin{center}
		$h_{1}(\left | s \right |) \leq \left | h(s) \right | \leq h_{1}^{-1}(\left |s \right | )$,\ $\forall$  $\left | s \right | \leq  \epsilon,$ \\
		
		$c_{1}\left | s \right |\leq \left | h(s) \right |  \leq c_{2} \left | s \right |$,\ $\forall$ $\left | s \right | \geq  \epsilon$,
	\end{center}
	where $c_1,c_2,\epsilon$ are positive constants. Moreover, define $H$ to be a strictly convex $ C^{2} $ function in $ (0,r_2]$ for some $ r_2 \geq 0 $ such that $ H(s) = \sqrt{s}h_{1}(\sqrt{s}) $ when $ h_{1} $ is nonlinear. \medskip
	
	\label{eqn_s4}
	\item The constant $k$ in (\ref{eqn_s1}) is such that
	\begin{equation}
	0< k< k_{0}=\frac{ 2\pi le^{3}} {c_{p}},
	\end{equation}
	where $c_{p}$ is the smallest positive number satisfying
	\label{eqn_s5}
	\begin{equation}
	\left \| \nabla u \right \|_{2}^{2}\leq c_{p}\left \| \Delta u \right \|_{2}^{2},\,  u\in H_{0}^{2}(\Omega),
	\end{equation}
	where $\left \| . \right \|_{2}=\left \| . \right \|_{L^{2}(\Omega)}.$
\end{enumerate}
Throughout this article, we use $c$ to denote a generic positive constant.
\begin{remark}
	Hypothesis $(H3)$ implies that $\tau h(\tau)>0,\ \forall \tau>0.$  
\end{remark}

\begin{remark}
	If $B$ is a strictly convex $C^2$ function and strictly increasing on $(0,r]$ for some $r>0$, then we can extend $B$ to $\bar{B}.$ Moreover, $\bar{B} $ is also a strictly convex $C^2$ function and Strictly increasing on $(0,\infty)$ (see \cite{Mohammad_2019}). Similarly, we denote the extension of $H$ to be $\bar{H}$.
\end{remark}

\noindent Plate problems have been broadly explored by mathematicians and other scientists. This type of problems have a lot of applications in different areas of science and engineering such as material engineering, mechanical engineering, nuclear physics and optics. \\

Let us discuss some work related to the plate problems. In \cite{Benaissa_2006}, the authors treated the following problem

\begin{equation*}
\left\{\begin{array}{ll}
u_{tt}-\Delta u+\alpha(t)h(u_{t})=0, \hspace{0.5cm} \quad (x,t) \in \, \Omega\times(0,\infty)\\ 
u=0, \hspace{0.5cm} \textrm{on} \hspace{0.2cm} \, \partial\Omega\times(0,\infty),
\end{array}\right.
\end{equation*}

\noindent where $h$ is a function having a polynomial growth near the origin, and they have established few energy decay results. Decay results for arbitrary growth of the damping term have been considered and studied for the first time in the work of Lasiecka and Tataru (see \cite{Lasiecka_1993}). They have found that the energy decays as fast as the solution of an associated differential equation whose coefficients depend on the damping term. In \cite{Liu_2009}, Liu considered the following problem
\begin{equation*}
\left\{\begin{array}{ll}
\left |u_{t} \right |^{\rho }u_{tt}-\Delta u-\Delta u_{tt}-\int_{0}^{t}g(t-s)\Delta u(s) d s+\alpha(t) h(u_{t})=b\left | u \right |^{\rho-2}u,\medskip \\
\hspace{9cm}  {\rm in\ } \Omega \times(0,\infty), \medskip \\
u\left ( x,t \right )=0, \hspace{0.5cm} \textrm{on} \hspace{0.2cm} \, \partial \Omega \times(0,\infty), \medskip  \\
u(x,0)=u_{0}(x), \hspace{0.5cm}  u_{t}(x,0)=u_{1}(x) \hspace{0.5cm} \textrm{on} \hspace{0.2cm}\, \Omega,
\end{array} \right.
\end{equation*}

\noindent and they have proved a general decay result that depends on the behavior of $g,\alpha$ and $h$ without imposing any restrictive growth assumption on the damping term at origin. For more results in the direction of the plate problems, see \cite{di_shang_bvp, lasi_1992, lagnese_1989_book_plates, lagnese_book1989, M_rivera_1996, komornik_1994, Messaoudi_2002, ChenW_2009} and the references there in. \\

Now, let us review some work with a  logarithmic term that are related to the problem (\ref{eqn_s1}). Cazenave and Haraux \cite{caze_1980} studied the following problem:
\begin{equation}
\label{eqn_s6}
u_{tt}-\Delta u=u \ln |u|^{k}, \hspace{0.5cm} \textrm{in} \, \mathbb{R}^{3}
\end{equation}
and established the existence and uniqueness of the solution for the Cauchy problem. Gorka \cite{Gorka_2009} obtained the global existence of weak solutions in the one-dimensional case by using compactness arguments, for all
$(u_{0}, u_{1}) \in \, H_{0}^{1}([a,b]) \times L^{2}([a,b]),$
to the initial-boundary value problem (\ref{eqn_s6}). The authors in \cite{Bartkowski_2008} considered the one dimensional Cauchy problem for equation (\ref{eqn_s6}) and they have proved the existence of classical solutions and also they have investigated the weak solutions. Birula and Mycielski \cite{Bi-Bi_1975, Bi-Bi_1976} considered

\begin{equation*}
\left\{\begin{array}{ll}
u_{tt}-u_{xx}+u-\epsilon u \ln \left | u \right |^{2}=0 \, \hspace{0.5cm} \quad (x,t) \in [a,b] \times (0,T),  \medskip
\\
u\left ( a,t \right )=u\left ( b,t \right )=0, \hspace{0.5cm} \quad t \in \, (0,T), \medskip
\\
u(x,0)=u_{0}(x), \hspace{0.5cm}  u_{t}(x,0)=u_{1}(x) \hspace{0.5cm} \quad x \in \, [a,b] 
\end {array} \right.
\end{equation*}

\noindent which is a relativistic version of logarithmic quantum mechanics. Moreover, it can also be obtained for the p-adic string equation by taking the limit as $p \rightarrow 1$ (see \cite{Gorka_2011, Vladimirov_2005}). Mohammad M. Al-Gharabli \cite{Mohammad_2018-1} considered equation \eqref{eqn_s1} with out damping term and they have established the existence of solution and proved the decay rates and stability result. Mohammad M. Al-Gharabli $et.al.$ (in \cite{mohd_2021}) have  considered the viscoelastic problem with variable exponent and logarithmic nonlinearities:
\begin{equation*}
u_{tt}-\Delta u + u +\int_{0}^{t}b(t-s)\Delta u(s) d s+\left | u_{t} \right |^{\gamma (\cdot )-2}u_{t}=u\ln \left | u \right |^{\alpha }	
\end{equation*}
\noindent and they have established a global existence result using the
well-depth method and then they have also established explicit and general decay results under a wide class of relaxation functions. Gongwei Liu (\cite{liu_2020}) considered the differential equation
\begin{equation*}
u_{tt}+\Delta ^{2}u+\left | u_{t} \right |^{m-2}u_{t}=u\left | u \right |^{p-2} \log\left | u \right |^{k} \hspace{0.5cm} (x,t) \in \Omega \times \mathbb{R}^{+} 
\end{equation*}
\noindent with the bounday conditions given in \eqref{eqn_s1}. They have established the local existence result by the fixed point techniques. The global existence and decay estimate of the solution at sub-critical initial energy is obtained, and they additionally prove that the solution with negative initial energy blows up in finite time under some suitable conditions. Moreover, they find out the blow-up in finite time of solution at the arbitrarily high initial energy for linear damping ($i.e.,$ $m = 2$). In \cite{adel_bvp_2020}, Adel M. Al-Mahdi considered viscoelastic plate equation with infinite memory and logarithmic nonlinearity:
\begin{equation*}
\left\{ \begin{array}{ll}
\left |u_{t} \right |^{\rho }u_{tt}+\Delta ^{2}u+\Delta ^{2}u_{tt}+u-\int_{0}^{\infty}b(s)\Delta ^{2}u(t-s)ds=\alpha u \ln \left | u \right |,\medskip \\
\hspace{9cm}  {\rm in\ } \Omega \times(0,\infty), \medskip \\
u\left ( x,t \right )=\frac{\partial u}{\partial \nu}\left ( x,t \right )=0, \hspace{0.5cm} \text{in} \hspace{0.2cm} \, \partial \Omega \times(0,\infty), \medskip \\
u(x,-t)=u_{0}(x), \hspace{0.5cm}  u_{t}(x,0)=u_{1}(x) \hspace{0.5cm} \textrm{in} \hspace{0.2cm}\, \Omega.
\end {array} \right.
\end{equation*}
By  imposing minimal conditions on the relaxation function the authors in \cite{adel_bvp_2020} established an explicit and general decay rate results. See \cite{mohd_bvp_2019, kafini_salim_2018} (and the references there in) for more results in this direction.\\

\noindent In this article, we are engaged with the global existence and stability of the plate problem (\ref{eqn_s1}) with kernels $b$ having an arbitrary growth at infinity. This article organized as follows. In Section $2$, we establish the local existence of the solutions to the problem (\ref{eqn_s1}). The global existence is proved in Section $3$. Finally, in the last section we derive few stability results.

\section{Local existence}

In this section, we state and prove the local existence result for the problem (1). The energy associated with problem (1) is

\begin{equation}
\begin{array}{ll}
\label{E_fn}
E(t)& =\frac{1}{\rho +2}\left \| u_{t} \right \|_{\rho +2}^{\rho +2}+\frac{1}{2}\Big[(1-\int_{0}^{t}b(s) ds)\left \| \Delta u \right \|_{2}^{2}+\left \| \Delta u_{t} \right \|_{2}^{2}\medskip \\
& \hspace{1cm} -k\int_{\Omega }u^{2} \ln \left | u \right |\mathrm{d}x + \left \| u \right \|_2^2 + b\circ\Delta u\Big]+ \frac{k}{4}\left \| u \right \|_{2}^{2}.
\end{array}
\end{equation}

\noindent where the product $\circ$ is defined by

\begin{center}
	$(b\circ\Delta u)(t)=\int_{0}^{t}b(t-s)\left \| \Delta u(s)-\Delta u(t) \right \|_{2}^{2}ds.$
\end{center}
Direct differentiation of (\ref{E_fn}) with respect to $t$ and using (\ref{eqn_s1}) we observe that 
\begin{equation}
\begin{array}{ll}
\label{DE_fn}
E'(t)& =\frac{1}{2}(b'\circ\Delta u)(t)-\frac{1}{2}b(t)\left \| \Delta u \right \|_{2}^{2}-\int_{\Omega }u_{t}h(u_{t}) \medskip \\
& \leq \frac{1}{2}(b'\circ\Delta u)(t)-\int_{\Omega }u_{t}h(u_{t}) \medskip \\
& \leq 0.
\end{array}
\end{equation}

\begin{lemma}
	(Logarithmic Sobolev inequality) Let $u \in  H_{0}^{1}(\Omega) $ and $ a>0 $ be any number. Then 
	\begin{equation}
	\int_{\Omega}u^{2}\ln \left | u \right |dx\leq \frac{1}{2}\left \| u \right \|_{2}^{2} \ln\left \| u \right \|_{2}^{2}+\frac{a^{2}}{2\pi}\left \| \bigtriangledown u \right \|_2^{2}-(1+ \ln a)\left \| u \right \|_{2}^{2}
	\end{equation}
\end{lemma}

\begin{corollary}
	Let $u \in  H_{0}^{2}(\Omega) $ and $ a>0 $ be any number. Then
	
	\begin{equation}
	\int_{\Omega}u^{2} \ln \left | u \right |dx\leq \frac{1}{2}\left \| u \right \|_{2}^{2}ln\left \| u \right \|_{2}^{2}+\frac{c_{p}a^{2}}{2\pi}\left \| \Delta  u \right \|_2^{2}-(1+ \ln a)\left \| u \right \|_{2}^{2}
	\end{equation} 
\end{corollary}

\begin{lemma}
	\label{eps_1}
	Let $ \epsilon _{0} \in (0,1), $ then there exists $ d_{\epsilon _{0}}>0 $ such that 
	\begin{equation}
	s\left | \ln s \right |\leq s^{2}+d_{\epsilon _{0}}s^{1-\epsilon _{0}}  \forall s> 0
	\end{equation}
\end{lemma}
\begin{definition}
	A function 
	\begin{center}
		$u\in C^{1}([0,T],H_{0}^{2}(\Omega ))$
	\end{center}
	is called a weak solution of (\ref{eqn_s1}) on $[0,T]$ if, for any $t\in [0,T]$ and $\forall w\in H_{0}^{2}(\Omega)$, $u$ satisfies
	
	\begin{equation}
	\label{weak_sol}
	\left\{ \begin{array}{ll}
	\displaystyle\int_{\Omega}\!\left | u_{t}\right |^{\rho}u_{tt}(x,t)w(x) dx\!+\!\int_{\Omega}\!\!\Delta u(x,t)\Delta w(x) dx\!+\!\int_{\Omega}\!\! \Delta u_{tt}(x,t)\Delta w(x) dx \medskip \\
	\hspace{1cm} +\displaystyle\int_{\Omega} u(x,t)w(x) dx-\displaystyle\int_{\Omega}\Delta w(x)\int_{0}^{t}b(t-s)\Delta u(s) ds \medskip \\
	\hspace{1cm} +\displaystyle\int_{\Omega}h(u_{t}(x,t))w(x) dx= k\int_{\Omega} u(x,t)w(x) \ln(|u(x,t)|) dx, \\ \\
	u(x,0)=u_{0}(x), \hspace{0.5cm} u_{t}(x,0)=u_{1}(x).
	\end {array} \right.
	\end{equation}
	
\end{definition}

\noindent {\bf Theorem}
Assume that the hypothesis $(H1)-(H4)$ hold. Let $(u_{0},u_{1}) \in H_{0}^{2}(\Omega )\times H_{0}^{2}(\Omega)$. Then the problem (\ref{eqn_s1}) has weak solution on $[0,T]$.\\

\begin{proof}
	To prove the existence of a solution to the problem (\ref{eqn_s1}), we use the Faedo-Galerkin approximations. Let $({w_j})_{j=1}^{\infty}$ be an orthogonal basis of the separable space $H_{0}^{2}(\Omega)$. Let $V_{m}$ = span$({w_{1}, w_{2},....,w_{m}})$ and let the projections of the initial data on the finite dimensional subspace $V_{m}$ be given by 
	\begin{equation}
	u_{0}^{m}(x)=\sum_{j=1}^{m}\alpha_{j}w_{j}(x), \hspace{0.5cm}   u_{1}^{m}(x)=\sum_{j=1}^{m}\beta_{j}w_{j}(x).\\
	\end{equation}
	
	\noindent We search for an approximation solution
	\begin{equation}
	u_{0}^{m}(x)=\sum_{j=1}^{m}g_{j}^{m}(t)w_{j}(x),
	\end{equation} 
	of the approximate problem in $V_{m}:$
	\begin{multline}
	\left\{\begin{array}{ll}
	\label{KE_fn}
	\displaystyle\int_{\Omega}\Big[\left | u_{t}^{m} \right |^{\rho}u_{tt}^{m}w+\Delta u^{m}\Delta w+\Delta u_{tt}^{m} \Delta w+u^{m}w+h(u_{t}^{m})w \medskip \\
	\hspace{0.5cm}-\displaystyle\int_{0}^{t}b(t-s)\Delta u^{m}(s)\Delta wds \Big] dx=k\int_{\Omega}wu^{m} \ln \left | u^{m} \right |dx, \hspace{0.1cm} \forall w \in V_{m},\medskip \\	u^{m}(0):=u_{0}^{m}=\sum_{j=1}^{m}(u_{0},w_{j})w_{j}, \medskip \\ u_{t}^{m}(0):=u_{1}^{m}=\sum_{j=1}^{m}(u_{1},w_{j})w_{j}.
	\end{array} \right .
	\end{multline}
	
	\noindent This gives a system of ordinary differential equation (ODE's) for the unknown functions $g_{j}^{m}(t)$. Using the standard existence theory for ODE's, one can obtain functions
	\begin{center}
		$g_{j}:[0,t_{m}) \rightarrow \mathbb{R},  j=1,2,....m,$
	\end{center}
	\noindent which satisfy (\ref{KE_fn}) in a maximal interval $[0,t_{m}), t_{m} \in (0,T]$. Later, we show that $ t_{m}=T $ and the local solution is uniformly bounded which is independent of $m$ and $t$. To do this, substitute $w=u_{t}^{m}$ in (\ref{KE_fn}) and using integration by parts to obtain 
	\begin{equation}
	\label{KE_fn1}
	\frac{d }{dt}E^{m}(t)\leq \frac{1}{2}(b'o\Delta u^{m})-\int_{\Omega}u_{t}^{m}h(u_{t}^{m}) dx\leq 0,
	\end{equation}
	where
	\begin{multline}
	E^{m}(t)=\frac{1}{\rho +2}\left \| u_{t}^{m} \right \|_{\rho +2}^{\rho +2}+\frac{1}{2}((1-\int_{0}^{t}b(s)ds)\left \| \Delta u^{m} \right \|_{2}^{2}+\left \| \Delta u_{t}^{m} \right \|_{2}^{2}
	\\
	-k\int_{\Omega }\left |u^{m}  \right |^{2} \ln \left | u^{m} \right | dx)+\frac{k+2}{4}\left \| u^{m} \right \|_{2}^{2}+\frac{1}{2}(b\circ\Delta u^{m}),
	\end{multline}
	
	\noindent from (\ref{KE_fn1}), we have 
	\begin{center}
		$E^{m}(t) \leq E^{m}(0), \forall t\geq 0.$ 
	\end{center}
	
	\noindent the logarithmic Sobolev inequality together with last inequality, we observe that 
	
	\begin{multline}
	\| u_t^m \|^{\rho+2}_{\rho+2} + \| \Delta u_t^m \|_2^2 + \left( l- \frac{k a^2c_p}{2\pi} \right) \|\Delta u^m \|^2_2\\  + \left[ \frac{k+2}{2} + k (1+  \ln a) \right] \| u^m \|^2_2  + b\circ \Delta u^m 
	\leq 2E^m(0)+ \| u^m \|^2_2 \ln \| u^m \|^2_2,
	\end{multline}
	
	\noindent Choose $a$ such that 
	
	\begin{equation}
	e^{{-3}/{2}} < a < \sqrt{\frac{2\pi l}{kc_p}},
	\end{equation}
	
	\noindent then a satisfies
	
	\begin{equation}
	\label{assmup_a}
	l-\frac{ka^2 c_p}{2\pi} > 0,
	\end{equation} 
	
	\noindent and 
	\begin{equation}
	\frac{k+2}{2}+ k (1+ \ln a) > 0.
	\end{equation} 
	
	\noindent So, we obtain
	\begin{multline}
	\label{KE_fn2}
	\| u_t^m \|^{\rho+2}_{\rho+2} + \| \Delta u_t^m \|_2^2 + \| \Delta u^m \|_2^2 + \| u^m \|_2^2 \\ +  b\circ \Delta u^m 
	\leq c \left( 1+ \| u^m \|^2_2  \ln \| u^m \|^2_2  \right)
	\end{multline}
	\noindent And we know that 
	\begin{center}
		$u^{m}(.,t)=u^{m}(.,0)+\int_{0}^{t}\frac{\partial u^{m}}{\partial s}(.,s) ds.$
	\end{center}	
	\noindent Using Cauchy Schwarz inequality, observe that
	\begin{multline}
	\left \| u^{m}(t) \right \|_{2}^{2}\leq 2\left \| u^{m}(0) \right \|_{2}^{2}+2\left \| \int_{0}^{t}\frac{\partial u^{m}}{\partial s}(s)ds \right \|_{2}^{2} \\ \leq 2\left \| u^{m}(0)\right \|_2+2T\int_{0}^{t}\left \| u_{t}^{m}(s) \right \|_{2}^{2} ds,
	\end{multline}
	\noindent therefore inequality (\ref{KE_fn2}) gives 
	\begin{equation}
	\label{KE_fn3}
	\| u^m \|^2_2 \leq 2 \| u^m(0) \|^2_2 + 2Tc \left( 1+ \int_0^t  \| u^m \|^2_2  \ln \| u^m \|^2_2 ds \right),
	\end{equation}
	\noindent if we substitute $c_1$ = max\{${2Tc,\ 2 \| u(0) \|^2_2}\},$ then (\ref{KE_fn3}) leads to
	\begin{equation*}
	\| u^m \|^2_2 \leq  2c_1 \left( 1+ \int_0^t  \| u^m \|^2_2  \ln (\| u^m \|^2_2 ) ds \right),
	\end{equation*}
	
	\noindent since $c_{1} \geq 0$, we get
	
	\begin{equation}
	\label{CS_1}
	\left \| u^{m} \right \|_2^2\leq 2 c_{1}\left ( 1+\int_{0}^{t} (c_{1}+\left \| u^{m} \right \|_{2}^{2}) \ln (c_{1}+\left \| u^{m} \right \|_{2}^{2})\right )ds.
	\end{equation}
	
	\noindent When Logarithmic Gronwall inequality applied to (\ref{CS_1}), we get the following estimate:
	\begin{center}
		$\left \| u^{m} \right \|_{2}^{2}\leq 2c_{1}e^{2c_{1}T}=c_{2}.$
	\end{center}
	\noindent Hence, from inequality (\ref{KE_fn2}) it follows that
	
	\begin{center}
		$\ (b\circ \Delta u^m)(t) + \| u_t^m  \|^{\rho+2}_{\rho+2} + \| \Delta u_t^m  \|^2_2 + \| \Delta u^m \|^2_2 + \| u^m \|^2_2 \leq c(1+c_{2} \ln c_{2})\leq c_3.$
	\end{center}
	\noindent This implies
	\begin{equation}
	\label{sup_1}
	^{sup}_{t \in (0, t_m)} \left[ (go \Delta u^m)(t) + \| u_t^m  \|^{\rho+2}_{\rho+2} + \| \Delta u_t^m  \|^2_2 + \| \Delta u^m \|^2_2 + \| u^m \|^2_2 \right] \leq c_3.
	\end{equation}
	
	\noindent So, we have
	\begin{multline}
	\left\{\begin{array}{ll}
	u^{m} \text {is uniformaly bounded in }  L^{\infty }(0,T; H_{0}^{2}(\Omega)) ,
	\\
	u_{t}^{m} \text{ is uniformaly bounded in }  L^{\infty }(0,T; L^{\rho+2}(\Omega)) \cap  L^{\infty }(0,T; H_{0}^{2}(\Omega)),
	\end{array} \right.
	\end{multline} 
	\noindent therefore, these satisfies a subsequence of $ (u_{m}) $, such that 
	\begin{multline}
	\label{u_bdd}
	\left\{\begin{array}{ll}
	u^{m}\overset {\ast }{\rightarrow} u \ {\rm in\ }  L^{\infty }(0,T; H_{0}^{2}(\Omega)) ,
	\\
	u_{t}^{m} \rightharpoonup u_{t} \ {\rm in\ }  L^{\infty }(0,T; L^{\rho+2}(\Omega)) \cap  L^{\infty }(0,T; H_{0}^{2}(\Omega)) ,
	\\
	u^{m}\rightharpoonup u_{t} \ {\rm in\ }  L^{2}(0,T; H_{0}^{2}(\Omega)) ,
	\\
	u^{m}\rightharpoonup u_{t} \ {\rm in\ }  L^{2}(0,T; L^{\rho+2}(\Omega))\cap L^{2}(0,T; H_{0}^{2}(\Omega)) ,
	\end{array} \right.
	\end{multline}
	\noindent where $ \overset {\ast }{\rightarrow} $ represent the weak $ \ast $ convergence and $ \rightharpoonup $ represent weak convergence. Therefore, the approximate solution is uniformly bounded and it is independent of $m$ and $t$. Therefore we can extend $t_{m}$ to $T$. \\
	
	Next we prove that $ u_{tt}^{m} $ is bounded in $ L^{2}(0,T;H_{0}^{2}(\Omega)) $. To do this, we substitute $ w=u_{tt}^{m} $ in (\ref{weak_sol}). Using (\ref{eqn_s1}), we see that 
	\begin{multline}
	\int_{\Omega}\left | u_{t}^{m} \right |^{\rho}\left | u_{tt}^{m} \right |^{2}dx+\left \| \Delta u_{tt}^{m} \right \|_{2}^{2}=-\int_{\Omega}(\Delta u^{m}\Delta u_{tt}^{m}+u^{m}u_{tt}^{m})
	\\
	+\int_{\Omega}\int_{0}^{t}b(t-s)\Delta u^{m}(s)\Delta u_{tt}^{m}(t)ds dx-\int_{\Omega}h(u_{t}^{m})u_{tt}^{m}dx+k\int_{\Omega}u_{tt}^{m}u^{m} \ln \left | u^{m} \right |dx.
	\end{multline}
	By using the Cauchy-Schwarz' inequality, Young's inequality, and the embedding inequality we obtain,\medskip \\
	$\displaystyle\int_{\Omega}\left | u_{t}^{m} \right |^{\rho}\left | u_{tt}^{m} \right |^{2}dx+\left \| \Delta u_{tt}^{m} \right \|_{2}^{2}$
	\begin{equation}
	\label{bdd_1}
	\begin{array}{ll}
	& \leq \delta \left \| \Delta u_{tt}^{m} \right \|_{2}^{2}   + \frac{1}{4\delta}\left \| \Delta u^{m}(t) \right \|_2^{2} + \delta \left \| u_{tt}^{m} \right \|_{2}^{2} + \frac{1}{4\delta}\left \| u^{m}(t) \right \|_2^{2} +\delta \left \| \Delta u_{tt}^{m} \right \|^{2} \medskip \\
	& \quad +\displaystyle\frac{1}{4\delta }\left ( \int_{0}^{t}b(t-s)\left \| \Delta u^{m}(s) \right \|ds \right )^{2}+ \frac{1}{4\delta}\int_{\Omega}h^2(u_{t}^{m})dx + \delta \left \| u_{tt}^{m} \right \|_{2}^{2}\medskip \\
	& \quad +k\int_{\Omega}u_{tt}^{m}u^{m} \ln \left | u^{m} \right |dx\medskip \\
	& \leq 2\delta \left \| \Delta u_{tt}^{m} \right \|_{2}^{2} + 2\delta \left \|  u_{tt}^{m} \right \|_{2}^{2}  +k\int_{\Omega}u_{tt}^{m}u^{m} \ln \left | u^{m} \right |dx\medskip \\
	& \quad +\displaystyle\frac{1}{4\delta}\Bigg[\left \| \Delta u^{m}(t) \right \|_2^{2} + \left ( \int_{0}^{t}b(t-s)\left \| \Delta u^{m}(s) \right \|ds \right )^{2} +  \int_{\Omega} h(u_{t}^{m}) dx + \|u^m\|^2 \Bigg] \medskip \\
	& \leq c\delta \left \| \Delta u_{tt}^{m} \right \|_{2}^{2} +k\int_{\Omega}u_{tt}^{m}u^{m} \ln \left | u^{m} \right |dx\medskip \\
	& \quad +\displaystyle\frac{1}{4\delta}\Bigg[\left \| \Delta u^{m}(t) \right \|_2^{2} + \left ( \int_{0}^{t}b(t-s)\left \| \Delta u^{m}(s) \right \|ds \right )^{2} +  \int_{\Omega} h(u_{t}^{m}) dx + \|u^m\|^2 \Bigg].
	\end{array}
	\end{equation}

	\noindent Using Lemma \ref{eps_1} with $\epsilon_0=\frac{1}{2}$, the second term in the right hand side of (\ref{bdd_1}) is estimated as follows:
	\begin{multline}
	\label{bdd_2}
	\begin{array}{ll}
	\!\!\!\!\!\!k\int_{\Omega}u_{tt}^{m}u^{m} \ln \left | u \right |^{m}dx\!\!\! & \leq c\int_{\Omega}u_{tt}^{m}\left ( \left | u^{m} \right |^{2}+d_{\frac{1}{2}}\sqrt{u^{m}} \right )dx
	\\
	& \leq c\left ( \delta \int_{\Omega}\left | u_{tt}^{m} \right |^{2}dx+\frac{1}{4\delta}\int_{\Omega}\left ( \left | u^{m} \right |^{2}+d_{\frac{1}{2}}\sqrt{u^m} \right )^{2} dx\right )
	\\
	& \leq c\delta\left \| u_{tt}^{m} \right \|_{2}^{2}+\frac{c}{4\delta}\left ( \int_{\Omega}\left | u^{m} \right |^{4}dx+\int_{\Omega}\left | u^{m} \right |dx \right )
	\\
	& \leq c\delta\left \|\Delta u_{tt}^{m}  \right \|_{2}^{2}+\frac{c}{4\delta}\left ( \left \| \Delta u^{m} \right \|_{2}^{4} +\left \| u^{m} \right \|_{2}\right ).
	\end{array} 
	\end{multline}
	\noindent from (\ref{bdd_1}) and (\ref{bdd_2}) we have\medskip \\
	\begin{equation}
	\label{bdd_2.}
	\begin{array}{ll}
	\displaystyle\int_{\Omega}\left | u_{t}^{m} \right |^{\rho}\left | u_{tt}^{m} \right |^{2}dx+(1-c\delta)\left \| \Delta u_{tt}^{m} \right \|_{2}^{2} 
	\leq \frac{c}{4\delta}\left ( \left \| \Delta u^{m} \right \|_{2}^{4} +\left \| u^{m} \right \|_{2}\right ) \medskip \\ + \displaystyle\frac{1}{4\delta}\Bigg[\left \| \Delta u^{m}(t) \right \|^{2} + \left ( \int_{0}^{t}b(t-s)\left \| \Delta u^{m}(s) \right \|ds \right )^{2} +  \int_{\Omega} h(u_{t}^{m}) dx + \|u^m\|^2 \Bigg]
	\end{array} 
	\end{equation}
	
	\noindent Now we prove that $ h(u_{t}^{m} ) $ is bounded in $ L^{2}(0,T;L^{2}(\Omega)).$ For this purpose, we consider two cases:
	\\
	Case {1}. When $ h_{1} $ is linear, we can directly prove that 
	\begin{center}
		$ \int_{0}^{t}\int_{\Omega}h^{2}(u_{t}^{m}) dx dt\leq C $
	\end{center}
	Case {2}. When $ h_{1} $ is nonlinear,
	\\
	Consider $ \left | s \right |\leq \epsilon  $, then\\
	$$ H(h^{2}(s))=\left | h(s) \right |h_{1}(\left | h(s) \right |)\leq sh(s), $$ 
	this implies that 
	$$ H^{-1}(sh(s))\geq h^{2}(s), \hspace{0.5cm} \forall \left | s \right |\leq \delta. $$ 
	then using the similar lines from Theorem $3.2$ in \cite{Mohammad_2018}, we get
	$ \int_{0}^{t}\int_{\Omega}h^{2}(u_{t}^{m}) dx dt\leq C_T.$ we conclude that $ h(u_{t}^{m}) $ is bounded in $ L^{2}(0,T;L^{2}(\Omega)) $.\\
	\noindent Integrating (\ref{bdd_2.}) from $ (0, T) $, using the hypothesis  $(H1)$ and (\ref{sup_1}), we obtain
	\begin{multline}
	\label{bdd_3}
	\int_{0}^{T}\int_{\Omega}\left | u_{t}^{m} \right |^{\rho}\left | u_{tt}^{m} \right |^{2}dxdt+(1-2c\delta)\int_{0}^{T}\left \| \Delta u_{tt}^{m} \right \|_{2}^{2}dt\\
	\leq \frac{c}{\delta}\int_{0}^{T}\left [ (b\circ\Delta u^{m})(t)+\left \| \Delta u^{m} \right \|_{2}^{2}+\left \| \Delta u^{m} \right \|_{2}^{4}+\left \| u^{m} \right \|_{2} \right ]dx\\-c_{\delta} \int_{0}^{T} \int_{\Omega}h^{2}(u_{t}^{m}) dxdt.
	\end{multline}
	\noindent From (\ref{bdd_3}) and using the fact that $ h(u_{t}^{m} ) $ is bounded in $ L^{2}(0,T;L^{2}(\Omega))$, it is easy to observe that for $ \delta $ small enough, 
	\begin{equation}
	\label{u_tt}
	u_{tt}^{m}\ \ {\rm is\ bounded\ in\ }  L^{2}(0,T;H_{0}^{2}(\Omega)). 
	\end{equation}
	
	\noindent Taking $ m \rightarrow \infty $  to (\ref{weak_sol}) and from (\ref{u_bdd}) and (\ref{u_tt}) (thanks to Aubin-Lions' theorem), for all $ w \in H_{0}^{2}(\Omega) $ and  $a.e.\ t \in (0,T) $, we see that
	\begin{multline*}
	\int_{\Omega}\left | u_{t} \right |^{\rho}u_{tt}wdx+\int_{\Omega}\Delta u\Delta wdx+\int_{\Omega}\Delta u_{tt}\Delta wdx+\int_{\Omega}uwdx+\int_{\Omega}h(u_t)wdx
	\\
	-\int_{\Omega}\Big(\int_{0}^{t}g(t-s)\Delta u(s)ds\Big)\Delta w dx=k\int_{\Omega}wu\ln \left | u \right |dxds.
	\end{multline*}
	This complete the proof of this theorem.
\end{proof}

\section{Global existence}
In this section, under smallness condition on the initial data, we state and prove global existence result. For the sake of simplicity, we introduce the following functionals:
\begin{multline}
\label{KE_fn4}
J(t): = J(u(t)) \\ =\frac{1}{2}\left [ \left \| \Delta u_{t} \right \|_{2}^{2}+(1-\int_{0}^{t}b(s)ds)\left \| \Delta u \right \|_{2}^{2}+\left \| u \right \|_{2}^{2}+(b\circ\Delta u)-\int_{\Omega}u^{2} \ln \left | u \right |^{k} \right ]\\+\frac{k}{4}\left \| u \right \|_{2}^{2},
\end{multline} 
and
\begin{equation}
\label{KE_fn5}
I(t):=I(u(t))
=\left \| \Delta u_{t} \right \|_{2}^{2}+(1-\int_{0}^{t}b(s)ds)\left \| \Delta u \right \|_{2}^{2}+\left \| u \right \|_{2}^{2}+(b\circ\Delta u)-\int_{\Omega}u^{2} \ln \left | u \right |^{k},
\end{equation}

\noindent \textbf{Note:} From (\ref{E_fn}), (\ref{KE_fn4}) and (\ref{KE_fn5}), it is clear that
\begin{equation} 
J(t)=\frac{1}{2}I(t)+\frac{k}{4}\left \| u \right \|_{2}^{2}
\end{equation}
\begin{equation}
\label{reEJ_1}
E(t)=\frac{1}{\rho+2 }\left \| u_{t} \right \|_{\rho+2}^{\rho+2}+J(t),
\end{equation}

\noindent \textbf{Notation:}
Define $ \overline{\rho}=e^{\frac{2Q_{0}-k}{k}} $,\ $ d=\frac{1}{2}Q_{0}\overline{\rho}^{2}-\frac{k}{4}\overline{\rho}^{2}ln\overline{\rho}^{2} $ and\medskip\\ $ Q_{0}=\frac{k+2}{2}+k(1+ \ln a), $ where $ 0<a<\sqrt{\frac{2\pi c_{p}l}{k}}. $
\begin{lemma}
	\label{Global_1}
	Let $(u_{0},u_{1} ) \in H_{0}^{2}(\Omega)\times H_{0}^{2}(\Omega)$ and assume that the hypothesis $(H1)$ holds. Further assume that $ \left \| u_{0} \right \|< \overline{\rho } $ and  	$ 0 < E(0) <d.$ Then $ I(u)>0 \hspace{0.2cm} \forall t\in [0,T).$
\end{lemma}
\begin{proof}
	We divide the proof for this lemma into two steps. In step (1), we prove that $ \left \| u \right \|_{2}< \overline{\rho } \hspace{0.2cm} \forall \ t \in [0,T) $ and in step (2), we prove $ I(t) >0 \hspace{0.2cm} t \in [0,T).$
	\\
	\textbf{Step 1.} From (\ref{E_fn}) and (\ref{KE_fn4}), and using Logarithmic Sobolev inequality it is easy to see that
	\begin{multline*}
	E(t)\geq J(t)
	\\
	\geq \frac{1}{2}\left ( l-\frac{c_{p}ka^{2}}{2\pi} \right )\left \| \Delta u \right \|^{2}+\frac{1}{2}\left ( \frac{k+2}{2}+k(1+ \ln a)-\frac{k}{2} \ln \left \| u \right \|_{2}^{2} \right )\left \| u \right \|_{2}^{2}.
	\end{multline*}
	Using (\ref{assmup_a}), we obtain 
	\begin{equation}
	\label{relEQ_1}
	E(t)\geq Q(\tilde{\rho})=\frac{1}{2}Q_{0}\tilde{\rho}^{2}-\frac{k}{4}\tilde{\rho}^{2}\ln \tilde{\rho}^{2}
	\end{equation}
	where $ \tilde{\rho}=\left \| u \right \|_{2}.$ Observe that 
	(\ref{relEQ_1}) implies $ Q $ is increasing on $ (o, \overline{\rho }) $ and decreasing on $ (\overline{\rho }, \infty) $. Also, note that $ Q(\tilde{\rho }) \rightarrow -\infty $ as $\tilde{\rho } \rightarrow +\infty $. Let 
	\begin{center}
		$\underset{0<\tilde{\rho}<+\infty}{max}Q(\tilde{\rho})=\frac{1}{2}Q_{0}\overline{\rho}^{2}-\frac{\alpha}{4}\overline{\rho}^{2}ln \overline{\rho}^{2}=Q(\overline{\rho})=d. $
	\end{center}
	\noindent Suppose that $ \left \| u \right \|< \overline{\rho} $ does not hold in $ [0, T), $ then there exist $ t_{0} \in (0,T) $ and $ \left \| u(x, t_{0}) \right \|=\overline{\rho}. $ Using (\ref{relEQ_1}), we get $ E(t_{0})\geq Q(\left \| u(x, t_{0}) \right \|_{2})=Q(\overline{\rho})=d $, which is contradiction to the fact $ E(t) \leq E(0) <d \hspace{0.2cm} \forall t>0 $. Hence $ \left \| u \right \|_{2}<\overline{\rho} \hspace{0.2cm} \forall t \in [0,T) $. Hence, $ \left \| u \right \|< \overline{\rho} $ for all $t\in[0,T).$\medskip\\
	\textbf{Step 2.}
	Using the definition of $ I(t) $ and (\ref{assmup_a}), we notice that for $ t \in [0,T), $
	\begin{equation*}
	\begin{array}{ll}
	I(t)& \geq \left ( l-\frac{c_{p}ka^{2}}{2\pi} \right )\left \| \Delta u \right \|_{2}^{2}+\left ( 1+k(1+ \ln a)-\frac{k}{2} \ln \left \| u \right \|_{2}^{2} \right )\left \| u \right \|_{2}^{2}
	\medskip \\
	& \geq \left ( l-\frac{c_{p}ka^{2}}{2\pi} \right )\left \| \Delta u \right \|_2^{2}+\left \| u \right \|_2^{2} \medskip \\
	& \geq 0	.
	\end {array} 
	\end{equation*}
	\noindent This complete the proof of this lemma.
\end{proof}

\begin{remark}
	\label{rmk_global}
	Under the assumptions of Lemma (\ref{Global_1}), for $ t \in [0,T), $ we have
	\begin{center} 
		$ \left \| u_t \right \|_{\rho+2}^{\rho+2}\leq (\rho+2) E(t) \leq (\rho+2)E(0),$
	\end{center}
	and 
	\[
	\left \| \Delta u_t \right \|^{2}\leq 2E(t) \leq 2E(0).
	\]
	This shows that the solution is global and bounded in time (in the above mentioned norm).
\end{remark}

\section{Stability}
In this section, we state and prove the decay of the solutions of the problem (\ref{eqn_s1}). At first we establish some lemmas which are useful to prove our main results. 

\begin{lemma}
	\label{est_b1}
	Assume that $b$ satisfies the hypothesis $(H1)$. Then, for $u\in H_0^2(\Omega)$, we have
	\[
	\int_\Omega \Big( \int_0^t b(t-s)(u(t)-u(s)) ds\Big)^2dx \leq c(b\circ \Delta u)(t),
	\]
	and
	\[
	\int_\Omega \Big( \int_0^t b'(t-s)(u(t)-u(s)) ds\Big)^2dx \leq -c(b'\circ\Delta u)(t),
	\]
	for some $c>0$.
\end{lemma}
\noindent For the proof of Lemma \ref{est_b1}, refer to \cite{Mohammad_2018-1}.

\begin{lemma}
	\label{est_h1}
	Assume that $h$ satisfies the hypothesis $(H3)$. Then, the solution of \eqref{eqn_s1} satisfies the following estimates:
	\begin{equation}
	\displaystyle \int_{\Omega} h^2(u_t)dx \leq c \int_{\Omega}u_th(u_t)dx \leq -cE'(t), \hspace{1cm} {\rm if}\ h_1 \ {\rm is\ linear}
	\end{equation}
	\begin{equation}
	\displaystyle \int_{\Omega} h^2(u_t)dx \leq c H^{-1}(G(t)) - cE'(t), \hspace{1cm} {\rm if}\ h_1 \ {\rm is\ nonlinear}
	\end{equation}
	where $G(t):=\frac{1}{|\Omega_1|} \int_{\Omega_1} u_th(u_t)dx \leq -cE'(t)$ and $\Omega_1=\{ x\in \Omega: |u_t|\leq \epsilon \}$ for some $c, \epsilon >0$.
\end{lemma}

\begin{lemma}
	\label{est_b_t}
	Assume that $b$ satisfies the hypothesis $(H1)$ and $(H2)$. Then the solution of \eqref{eqn_s1} satisfies the following estimate:
	\begin{equation}
	\displaystyle \int_{t_1}^t b(s)\int_{\Omega} |\Delta u(t)- \Delta u(t-s)|^2dxds \leq \frac{t-t_1}{\delta} \bar{B}^{-1}\Big( \frac{\delta M(t)}{(t-t_1)\xi(t)} \Big),
	\end{equation}	
	where $\delta \in (0,1)$, $\bar{B}$ is an extension of $B$ and 
	\begin{equation}
	\label{rel_M_E}
	M(t):= -\displaystyle \int_{t_1}^t b'(s)\int_{\Omega} |\Delta u(t)- \Delta u(t-s)|^2dxds \leq -cE'(t).
	\end{equation}
\end{lemma}

\noindent The proofs of Lemma's \ref{est_h1} and \ref{est_b_t} follows from the similar lines as in  \cite{Mohammad_2019}. So, we skip the proof's.

\begin{lemma}
	Under the hypothesis $(H1)-(H4)$, the functional $\Psi_1(t)$
	\[
	\Psi_1(t):=\frac{1}{\rho +1}\int_{\Omega} |u_t|^{\rho}u_tudx + \int_{\Omega} \Delta u\Delta u_tdx,
	\]
	satisfies the estimate:
	\begin{multline}
	\label{est_psi1}
	\Psi_1'(t) \leq \frac{1}{\rho +1} \|u_t\|_{\rho+2}^{\rho+2} + \|\Delta u_t\|_2^2 -\frac{l}{4}\|\Delta u\|_2^2 - \|u\|_2^2+c\int_{\Omega} h^2(u_t)dx \\ + c(b\circ \Delta u)(t) + k\int_{\Omega} u^2\ln |u|dx.
	\end{multline}
\end{lemma}

\begin{proof}
	Differentiating $\Psi_1$ with respect to $t$ and using (\ref{eqn_s1}), we get
	\begin{multline}
	\label{dummy_psi1}
	\Psi_1'(t) = \frac{1}{\rho +1} \|u_t\|_{\rho+2}^{\rho+2} + k\int_{\Omega} u^2\ln |u|dx - \|\Delta u\|_2^2 - \|u\|_2^2 +  \|\Delta u_t\|_2^2 \\ + \int_{\Omega} \Delta u(t) \int_0^t b(t-s) \Delta u(s)dsdx - \int_{\Omega} uh(u_t)dx.
	\end{multline}
	At first, we estimate the sixth term in the right hand side of (\ref{dummy_psi1}). Using the hypothesis $(H1)$ and Lemma \ref{est_b1}, we have\medskip \\
	$\int_{\Omega} \Delta u(t) \int_0^t b(t-s) \Delta u(s)dsdx$
	\begin{equation*}
	\begin{array}{ll}
	& = \int_{\Omega} \Delta u(t) \int_0^t b(t-s) (\Delta u(s) - \Delta u(t))dsdx + \int_{\Omega} (\Delta u(t))^2 \int_0^t b(t-s) ds \medskip \\
	& \leq \frac{\eta}{2} \int_{\Omega} |\Delta u|^2 + \frac{1}{2\eta} \int_{\Omega} \big( \int_0^t b(t-s)|\Delta u(t) - \Delta u(s)|ds\big)^2dx + (1-l)\int_{\Omega} |\Delta u|^2\medskip \\
	& = (1-l+\frac{\eta}{2}) \int_{\Omega} |\Delta u|^2 + \frac{c}{\eta} (b\circ \Delta u)(t) \medskip \\
	& = (1-\frac{l}{2}) \int_{\Omega} |\Delta u|^2 + \frac{c}{\eta} (b\circ \Delta u)(t).
	\end {array} 
	\end{equation*}
	The last line in the above inequality is obtained by choosing $\eta=l$. Now, we estimate the last term in (\ref{dummy_psi1}).
	\begin{equation*}
	\begin{array}{ll}
	\int_{\Omega} |uh(u_t)|
	& \leq \delta c\|\nabla u\|_2^2 + \frac{c}{\delta} \int_{\Omega}h^2(u_t)dx \medskip \\
	& \leq \delta c\|\Delta u\|_2^2 + \frac{c}{\delta} \int_{\Omega}h^2(u_t)dx.
	\end {array} 
	\end{equation*}
	Now choose $\delta = \frac{l}{4c}$ to get (\ref{est_psi1}). Hence proved.
\end{proof}

\begin{lemma}
	Under the hypothesis $(H1)-(H4)$, the functional $\Psi_2(t)$
	\[
	\Psi_2(t):=- \displaystyle \int_{\Omega} \Big( \Delta^2 u_t + \frac{1}{\rho +1}|u_t|^{\rho}u_t \Big) \int_0^t b(t-s) (u(t)-u(s))dsdx,
	\]
	satisfies the estimate:
	\begin{multline}
	\label{est_psi2}
	\Psi_2'(t) \leq -\frac{1}{\rho +1} \big(\int_0^t b(s)ds\big) \|u_t\|_{\rho+2}^{\rho+2} + \big( \frac{\delta}{2} + \delta_1 +2\delta_1(1-l)^2 \big) \|\Delta u\|_2^2 + \\ \big( \delta + c\delta_2(E(0))^{\rho} -\int_0^t b(s)ds\big) \|\Delta u_t\|_2^2 - \frac{c}{\delta} (b'\circ \Delta u)(t) + \\
	\big( c+ \frac{c}{\delta} + c\delta_1 + \frac{c}{\delta_1} \big) (b\circ \Delta u)(t)
	+ c(b\circ \Delta u)^{\frac{1}{1+\epsilon_0}}(t)+c\int_{\Omega} h^2(u_t)dx,
	\end{multline}
	for some $\epsilon_0\in (0,1).$
\end{lemma}

\begin{proof}
	Differentiating $\Psi_2$ with respect to $t$ and using (\ref{eqn_s1}), we get\\
	\begin{equation}
	\label{dummy_psi2}
	\begin{array}{ll}
	\!\!\Psi_2'(t)\!\! =\!\!
	& - \displaystyle\int_{\Omega}u\ln|u|^k\big( \int_0^t b(t-s)(u(t)-u(s))ds \big)dx \medskip \\
	&  + \displaystyle\int_{\Omega} \Delta u(t) \big(\int_0^t b(t-s) (\Delta u(t)-\Delta u(s))ds\big)dx  \medskip \\
	& + \displaystyle\int_{\Omega}u\big( \int_0^t b(t-s)(u(t)-u(s))ds \big)dx \medskip \\
	&- \displaystyle\int_{\Omega}\Big( \int_0^t b(t-s)\Delta u(s)ds \Big)\Big( \int_0^t b(t-s)(\Delta u(t)-\Delta u(s))ds \Big)dx \medskip \\
	& + \displaystyle\int_{\Omega}h(u_t)\big( \int_0^t b(t-s)(u(t)-u(s))ds \big)dx\medskip \\
	& - \displaystyle\int_{\Omega}\frac{1}{\rho+1}|u_t|^{\rho}u_t \big( \int_0^t b'(t-s)(u(t)-u(s))ds \big)dx \medskip \\
	& - \displaystyle\int_{\Omega}\frac{1}{\rho+1}|u_t|^{\rho}u_t^2 \big( \int_0^t b(s)ds \big)dx\medskip \\
	& -\displaystyle \int_{\Omega}\Delta u_t \big( \int_0^t b'(t-s)(\Delta u(t)-\Delta u(s))ds \big)dx \medskip \\
	& - \displaystyle\big( \int_0^t b(s)ds \big)\big(\int_{\Omega}|\Delta u_t|^2dx\big).
	\end {array} 
	\end{equation}
	Inorder to estimate the all the nine terms in the right hand side of (\ref{dummy_psi2}), we will use Young's inequality, Cauchy-Schwarz' inequality and Lemma \ref{est_b1}.  All the terms except for the third and fifth term, are estimated in the similar lines as in \cite{Mohammad_2018}. For the sake of completeness, we will write the final estimates.
	
	\begin{equation*}
	\begin{array}{ll}
	{\rm (I)\ }- \displaystyle\int_{\Omega}u\ln|u|^k\big( \int_0^t b(t-s)(u(t)-u(s))ds \big)dx 
	\medskip \\
	\hspace{4cm}\leq \displaystyle\frac{\delta}{4} \|\Delta u\|_2^2 + \frac{c}{\delta}(b\circ \Delta u)(t) + c(b\circ \Delta u)^{\displaystyle\frac{1}{1+\epsilon_0}}(t),
	\end {array} 
	\end{equation*}
	
	\begin{equation*}
	\begin{array}{ll}
	{\rm (II)\ }\displaystyle\int_{\Omega} \Delta u(t) \big(\int_0^t b(t-s) (\Delta u(t)-\Delta u(s))ds\big)dx \medskip \\
	\hspace{4cm}\leq 
	\delta_1 \|\Delta u\|_2^2 + \frac{c}{\delta_1}(b\circ \Delta u)(t),
	\end {array} 
	\end{equation*}
	
	\begin{equation*}
	\begin{array}{ll}
	{\rm (III)\ } \displaystyle\int_{\Omega}u\big( \int_0^t b(t-s)(u(t)-u(s))ds \big)dx \leq \frac{\delta}{4} \|\Delta u\|_2^2 + \frac{c}{\delta}(b\circ \Delta u)(t),
	\end {array} 
	\end{equation*}
	
	\begin{equation*}
	\begin{array}{ll}
	{\rm (IV)\ }- \displaystyle\int_{\Omega}\big( \int_0^t b(t-s)\Delta u(s)ds \big)\big( \int_0^t b(t-s)(\Delta u(t)-\Delta u(s))ds \big)dx \medskip \\
	\hspace{5cm}\leq \displaystyle\big( c\delta_1 + \frac{c}{\delta_1} \big) (b\circ \Delta u)(t) + 2\delta_1(1-l)^2 \|\Delta u\|_2^2,
	\end {array} 
	\end{equation*}
	
	\begin{equation*}
	\begin{array}{ll}
	{\rm (V)\ }\displaystyle\int_{\Omega}h(u_t)\big( \int_0^t b(t-s)(u(t)-u(s))ds \big)dx \medskip \\
	\hspace{5cm} \leq c\displaystyle\int_{\Omega} h^2(u_t)dx + (1-l)c^2(b\circ \Delta u)(t).
	\end {array} 
	\end{equation*}
	
	\begin{equation*}
	\begin{array}{ll}
	{\rm (VI)\ }- \displaystyle\int_{\Omega}\frac{1}{\rho+1}|u_t|^{\rho}u_t \big( \int_0^t b'(t-s)(u(t)-u(s))ds \big)dx  \medskip \\
	\hspace{5cm} \leq c\delta_2(E(0))^{\rho} \|\Delta u_t\|_2^2 - \displaystyle\frac{c}{\delta_2}(b'\circ \Delta u)(t),
	\end {array} 
	\end{equation*}
	
	\begin{equation*}
	\begin{array}{ll}
	{\rm (VII)\ }- \displaystyle\int_{\Omega}\frac{1}{\rho+1}|u_t|^{\rho}u_t^2 \big( \int_0^t b(s)ds \big)dx  = -\frac{1}{\rho +1} \big(\int_0^t b(s)ds\big) \|u_t\|_{\rho+2}^{\rho+2},
	\end {array} 
	\end{equation*}
	
	\begin{equation*}
	\begin{array}{ll}
	{\rm (VIII)\ }- \displaystyle\int_{\Omega}\Delta u_t \big( \int_0^t b'(t-s)(\Delta u(t)-\Delta u(s))ds \big)dx  \medskip \\
	\hspace{5cm} \leq \delta \|\Delta u_t\|_2^2 - \displaystyle \frac{c}{\delta}(b'\circ \Delta u)(t),
	\end {array} 
	\end{equation*}
	and
	\begin{equation*}
	\begin{array}{ll}
	{\rm (IX)\ }- \displaystyle\big( \int_0^t b(s)ds \big)\big(\int_{\Omega}|\Delta u_t|^2dx\big) = - \big( \int_0^t b(s)ds \big) \|\Delta u_t\|_2^2.
	\end {array} 
	\end{equation*}
	Combining all the estimates we will arrive at (\ref{est_psi2}).
\end{proof}

\begin{lemma}
	Assume that the hypothesis of Lemma \ref{Global_1} and the hypothesis of $(H1)-(H4)$ holds. Then there exists $N, \epsilon>0$ such that the functional $$L(t) = NE(t) + \epsilon \Psi_1(t) + \Psi_2(t)$$ satisfies, 
	\begin{equation}
	\label{rel_LE}
	L \sim E,
	\end{equation}
	and for all $t\geq 0,$ there exists $m>0$ and $\epsilon_0 \in (0,1)$ such that
	\begin{equation}
	\label{est_L'}
	L'(t)\leq -mE(t) + c(b\circ \Delta u)(t) + c(b\circ \Delta u)^{\frac{1}{1+\epsilon_0}}(t) + c\int_{\Omega} h^2(u_t)dx,
	\end{equation}
	
\end{lemma}

\begin{proof}
	In order to prove (\ref{rel_LE}), we use the Sobolev embedding $H_0^1(\Omega) \hookrightarrow L^{\rho+2}(\Omega)$,\ $\int_{\Omega} |u_t|^{2(\rho +1)}dx\leq c\|\Delta u_t\|^2$ and Remark \ref{rmk_global}. From \eqref{est_L'}, we observe that
	\begin{equation*}
	\begin{array}{ll}
	|L(t)-NE(t)| & \leq \epsilon|\Psi_1| + |\Psi_2| \medskip \\
	& \leq \frac{\epsilon}{\rho+1}\int_{\Omega} |u_t|^{\rho+1}|u|dx + |\epsilon\int_{\Omega } \Delta u\Delta u_t dx|\medskip \\
	& \hspace{1cm} + | \frac{1}{\rho +1}|u_t|^{\rho}u_t\int_0^t b(t-s) (u(t)-u(s))dsdx| \medskip \\
	& \hspace{1cm}  + | \frac{1}{\rho +1}(\Delta u_t)^{2} \int_0^t b(t-s) (u(t)-u(s))dsdx| \medskip \\
	& \leq \frac{\epsilon}{\rho+2}\|u_t\|_{\rho+2}^{\rho+2} + \frac{\epsilon}{(\rho+1)(\rho+2)}\|u\|_{\rho+2}^{\rho+2} + \frac{\epsilon}{2}\|\Delta u\|^2+\frac{\epsilon}{2}\|\Delta u_t\|_2^2 \medskip \\
	& \hspace{1cm} + \frac{1}{2(\rho+1)}\|u_t\|_{2(\rho+1)}^{2(\rho+1)} + \frac{1-l}{2(\rho+1)}c_p(b\circ \Delta u)(t)+ \frac{1}{2}\|\Delta u_t\|_2^2\medskip \\
	& \hspace{1cm} + \frac{1-l}{2}(b\circ \Delta u)(t)\medskip \\
	& \leq c(1+\epsilon)E(t).
	\end {array} 
	\end{equation*}
	Therefore, by choosing $N$ large enough we obtain (\ref{rel_LE}). For the inequality (\ref{est_L'}), we follow similar lines given in \cite{Mohammad_2018-1}.
\end{proof}

\begin{remark}
	\label{rmk_bu}
	Since, 
	\begin{center}
		$ E(t) \geq J(t) \geq \frac{1}{2}(b \circ \Delta u)(t) $ that imply
		\medskip \\
		$ (b \circ \Delta u)(t) \leq 2E(t) \leq 2E(0).$
	\end{center}
	Hence, 
	\begin{equation}
	(b \circ \Delta u)(t)\leq (b \circ \Delta u)^{\frac{\epsilon _{0}}{1+\epsilon _{0}}}(t)(b \circ \Delta u)^{\frac{1}{1+\epsilon _{0}}}(t)\leq c(b \circ \Delta u)^{\frac{1}{1+\epsilon _{0}}}(t).
	\end{equation}
\end{remark}
\begin{remark}
	\label{rmk_xi_bu}
	If b is linear then, we have 
	\begin{equation}
	\begin{array}{ll}
	\xi (t)(b \circ \Delta u)^{\frac{1}{1+\epsilon _{0}}}(t)& =\left [ \xi^{\epsilon _{0}}(t) \xi (t)(b \circ \Delta u)(t) \right ]^{\frac{1}{1+ \epsilon _{0}}}
	\medskip \\
	& \leq \left [ \xi^{\epsilon _{0}}(0) \xi (t) (b \circ \Delta u)(t)\right ]^{\frac{1}{1+\epsilon _{0}}}
	\medskip \\
	& \leq c\left ( \xi(t)(b \circ \Delta u) (t)\right )^{\frac{1}{1+\epsilon _{0}}} \medskip \\
	& \leq c(-E'(t))^{\frac{1}{1+\epsilon _{0}}}
	\end {array} 
	\end{equation}
\end{remark}

\begin{theorem}
	\label{energy_est_linear_h}
	Let $(u_0,u_1)\in H_0^2(\Omega)\times H_0^2(\Omega)$. Assume that the hypothesis $(H1)-(H4)$ holds and $h_1$ is linear. Then for all $t\geq t_0$, we have 
	\begin{equation}
	\label{energy_est1}
	{\rm\ if\ B\ is\ linear\ then\ }\ E(t) \leq c\Big( 1+ \int_{t_0}^{t}\xi^{1+\epsilon_0}(s)ds \Big)^{-\frac{1}{\epsilon_0}},\ \forall t\geq t_0,
	\end{equation}
	and
	\begin{equation}
	{\rm\ if\ B\ is\ nonlinear\ then\ }\ E(t) \leq c \big( t-t_0 \big)^{\frac{1}{1+\epsilon}} \mathcal{K}_1^{-1}\Big( \frac{c}{\big( t-t_0 \big)^{\frac{1}{1+\epsilon}} \int_{t_1}^t\xi(s)ds} \Big),\ \forall t\geq t_1,
	\end{equation}
	where $\mathcal{K}_1(t)=t\mathcal{K}'(\epsilon_1t)$ and $\mathcal{K}(t)=\Big( \big(\bar{B}^{-1}\big)^{\frac{1}{1+\epsilon}} \Big)^{-1}(t).$
\end{theorem}

\begin{proof} We will divide the proof of this theorem into two cases. \\
	\textbf{Case 1:} $G$ is linear. Using \eqref{est_L'},\\
	\begin{equation*}
	\begin{array}{ll}
	L'(t) &\leq  -mE(t) + c(b\circ \Delta u)(t) + c(b\circ \Delta u)^{\frac{1}{1+\epsilon_0}}(t) + c\int_{\Omega} h^2(u_t)dx\medskip \\
	& \leq -mE(t) + c(b\circ \Delta u)(t) + c(b\circ \Delta u)^{\frac{1}{1+\epsilon_0}}(t)+c(-E'(t)).
	\end {array} 
	\end{equation*}
	Denote $L_1(t)=L(t)+cE(t)$, then the above inequality becomes
	\begin{equation}
	\label{ineq_L1}
	L_1'(t) \leq  -mE(t) + c(b\circ \Delta u)(t) + c(b\circ \Delta u)^{\frac{1}{1+\epsilon_0}}(t)
	\end{equation}
	Multiplying $\xi(t)$ to (\ref{ineq_L1}) and using Remarks \ref{rmk_bu} and \ref{rmk_xi_bu}, and for all $t\geq t_0$ we get
	\begin{equation*}
	\xi(t)L_1'(t)\leq -m\xi(t)E(t) + c(-E'(t))^{\frac{1}{1+\epsilon_0}}.
	\end{equation*}
	Multiplying the above inequality by $\xi^{\epsilon_0}(t)E^{\epsilon_0}(t)$ and using the Young's inequality  and for any $\epsilon_1>0$ and $t\geq t_0$, we obtain
	\begin{equation*}
	\begin{array}{ll}
	\xi^{1+\epsilon_0}(t)E^{\epsilon_0}(t)L_1'(t) &\leq  -m\xi^{1+\epsilon_0}(t)E^{1+\epsilon_0}(t) + c(\xi E)^{\epsilon_0}(t)(-E'(t))^{\frac{1}{1+\epsilon_0}} \medskip \\
	& \leq -m\xi^{1+\epsilon_0}(t)E^{1+\epsilon_0}(t) + c\Big( \epsilon_1\xi^{1+\epsilon_0}(t)E^{1+\epsilon_0}(t)-c_{\epsilon_1}E'(t) \Big) \medskip \\
	& \leq -(m-\epsilon_1 c)\xi^{1+\epsilon_0}(t)E^{1+\epsilon_0}(t) - cE'(t).
	\end {array} 
	\end{equation*}
	Denote $L_2(t)=\xi^{1+\epsilon_0}(t)E^{\epsilon_0}(t)L_1(t)+cE(t)$ and choosing $\epsilon_1 <\frac{m}{c}$ and using the properties of $\xi$ and $E$ to get
	\begin{equation}
	\label{ineq_L2}
	L_2'(t)\leq -c\xi^{1+\epsilon_0}(t)E^{1+\epsilon_0}(t), \ \ \forall t\geq t_0,
	\end{equation}
	where $c_1 = m-\epsilon_1 c$. Since, $L_2 \sim E$ and from (\ref{ineq_L2}), it is easy to see that 
	\[
	E'(t)\leq -c\xi^{1+\epsilon_0}(t)E^{1+\epsilon_0}(t), \ \ \forall t\geq t_0,
	\] 
	Integrating the above inequality from $(t_0,t)$, we obtain (\ref{energy_est1}).\medskip \\
	\textbf{Case 2:} $G$ is nonlinear. Using \eqref{est_L'},\\
	\begin{equation*}
	\begin{array}{ll}
	L'(t) &\leq  -mE(t) + c(b\circ \Delta u)(t) + c(b\circ \Delta u)^{\frac{1}{1+\epsilon_0}}(t) + c\int_{\Omega} h^2(u_t)dx\medskip \\
	& \leq -mE(t) + c(b\circ \Delta u)(t) + c(b\circ \Delta u)^{\frac{1}{1+\epsilon_0}}(t)+c(-E'(t)).
	\end {array} 
	\end{equation*}
	Denote $L_1(t)=L(t)+cE(t)$, then using Lemma \ref{est_b_t} and Remark \ref{rmk_bu} the above inequality becomes
	\begin{equation*}
	\begin{array}{ll}
	L_1'(t) & \leq  -mE(t) + c(b\circ \Delta u)(t) + c(b\circ \Delta u)^{\frac{1}{1+\epsilon_0}}(t)\medskip \\
	& \leq -mE(t) + c(b\circ \Delta u)^{\frac{1}{1+\epsilon_0}}(t)\medskip \\
	& \leq -mE(t) + c(t-t_0)^{\frac{1}{1+\epsilon_0}} \bar{B}^{-1}\Big( \frac{\delta M(t)}{(t-t_0)\xi(t)} \Big)^{\frac{1}{1+\epsilon_0}}\medskip \\
	& \leq -mE(t) + c(t-t_0)^{\frac{1}{1+\epsilon_0}} \bar{B}^{-1}\Big( \frac{\delta M(t)}{(t-t_0)^{\frac{1}{1+\epsilon_0}}\xi(t)} \Big)^{\frac{1}{1+\epsilon_0}},
	\end {array} 
	\end{equation*}
	the last line in the above inequality follows from the fact that, there exists $t_1> t_0$ such that $\frac{1}{t-t_0}<1$ when ever $t>t_1$. Hence we have 
	$$\bar{B}^{-1}\Big( \frac{\delta M(t)}{(t-t_0)\xi(t)} \Big) \leq \bar{B}^{-1}\Big( \frac{\delta M(t)}{(t-t_0)^{\frac{1}{1+\epsilon_0}}\xi(t)} \Big)\ \ \forall t>t_1.$$
	Denote 
	\begin{equation}
	\label{notation}
	\mathcal{K}(t)=\Big[\Big( \bar{B}^{-1} \Big)^{\frac{1}{1+\epsilon_0}} \Big]^{-1}(t),\quad \alpha(t)=\frac{\delta M(t)}{(t-t_0)^{\frac{1}{1+\epsilon_0}}\xi(t)}.
	\end{equation}
	From the definition of $\mathcal{K}$, it is clear that $\mathcal{K}'(t)>0$ and $\mathcal{K}''(t)>0$ on $(0,r]$ where $r = \min\{r_1,r_2\}$. Using these notations, we obtain that for all $t\geq t_1$,
	\begin{equation}
	\label{new_esti_L1}
	L_1'(t) \leq  -mE(t) + c(t-t_0)^{\frac{1}{1+\epsilon_0}}\mathcal{K}^{-1}\big(\alpha(t)\big).
	\end{equation}
	Now, we define the functional $L_2$ as follows
	\begin{equation*}
	L_2(t) := \mathcal{K}'\Bigg(\frac{\epsilon_1}{(t-t_0)^{\frac{1}{1+\epsilon_0}}}\frac{E(t)}{E(0)}
	\Bigg)L_1(t),
	\end{equation*} 
	for some $\epsilon_1\in (0,r)$. Using the fact that $\mathcal{K}'(t)>0,$ $\mathcal{K}''(t)>0$ and $E'(t)\leq 0$ on $(0,r]$, and using \eqref{new_esti_L1} we obtain
	\begin{multline}
	\label{est_L2}
	L_2'(t) \leq  -mE(t)\mathcal{K}'\Bigg(\frac{\epsilon_1}{(t-t_0)^{\frac{1}{1+\epsilon_0}}}\frac{E(t)}{E(0)}\Bigg) + \\ c(t-t_0)^{\frac{1}{1+\epsilon_0}}\mathcal{K}'\Bigg(\frac{\epsilon_1}{(t-t_0)^{\frac{1}{1+\epsilon_0}}}\frac{E(t)}{E(0)}\Bigg) \mathcal{K}^{-1}\big(\alpha(t)\big), \forall t\geq t_1,
	\end{multline} 
	and 
	\begin{equation}
	L_2 \sim E.
	\end{equation}
	Let $\mathcal{K}^*$ denote the convex conjugate of $\mathcal{K}$ in the sense of Young (see \cite{young_ineq}), then we have 
	\[
	\mathcal{K}^*(\tau) = \tau (\mathcal{K'})^{-1}(\tau)-\mathcal{K}\big([\mathcal{K'}]^{-1}(\tau)\big), {\rm if\ } \tau\in (0,\mathcal{K'}(r)),
	\]
	here $\mathcal{K}^*$ satisfies the generalized Young Inequality:
	\[
	ab\leq \mathcal{K}^*(a)+\mathcal{K}(b),\ {\rm if\ } a\in (0,\mathcal{K'}(r)),\ b\in (0,r].
	\]
	So, by assuming $a=\mathcal{K}'\Bigg(\frac{\epsilon_1}{(t-t_0)^{\frac{1}{1+\epsilon_0}}}\frac{E(t)}{E(0)}\Bigg)$ and $b=\mathcal{K}^{-1}(\alpha(t))$ and using (\ref{est_L2}) we obtain
	\begin{equation*}
	\begin{array}{ll}
	L_2'(t) & \leq  -mE(t)\mathcal{K}'\Bigg(\frac{\epsilon_1}{(t-t_0)^{\frac{1}{1+\epsilon_0}}}\frac{E(t)}{E(0)}\Bigg)  \medskip \\
	& \quad + c(t-t_0)^{\frac{1}{1+\epsilon_0}}\mathcal{K}^*\Bigg[ \mathcal{K}'\Bigg(\frac{\epsilon_1}{(t-t_0)^{\frac{1}{1+\epsilon_0}}}\frac{E(t)}{E(0)}\Bigg) \Bigg] +  c(t-t_0)^{\frac{1}{1+\epsilon_0}}\alpha(t)\medskip \\
	& \leq -mE(t)\mathcal{K}'\Bigg(\frac{\epsilon_1}{(t-t_0)^{\frac{1}{1+\epsilon_0}}}\frac{E(t)}{E(0)}\Bigg) + c\epsilon_1\frac{E(t)}{E(0)}\mathcal{K}'\Bigg(\frac{\epsilon_1}{(t-t_0)^{\frac{1}{1+\epsilon_0}}}\frac{E(t)}{E(0)}\Bigg) \medskip \\
	& \quad + c(t-t_0)^{\frac{1}{1+\epsilon_0}}\alpha(t)
	\end {array} 
	\end{equation*} 
	Multiplying the above inequality by $\xi(t)$ and using \eqref{rel_M_E}and \eqref{notation} for all $t\geq t_1$ we get 
	\begin{multline*}
	\xi(t)L_2'(t) \leq -m\xi(t)E(t)\mathcal{K}'\Bigg(\frac{\epsilon_1}{(t-t_0)^{\frac{1}{1+\epsilon_0}}}\frac{E(t)}{E(0)}\Bigg) \medskip \\ +c\epsilon_1\xi(t)\frac{E(t)}{E(0)}\mathcal{K}'\Bigg(\frac{\epsilon_1}{(t-t_0)^{\frac{1}{1+\epsilon_0}}}\frac{E(t)}{E(0)}\Bigg) -cE'(t).
	\end{multline*}
	By setting $L_3:=\xi L_2+cE$ (notice that $L_3\sim E$), we get for all $t\geq t_1$
	\[
	L_3'(t) \leq -\big( mE(0)-c\epsilon_1 \big)\xi(t)\frac{E(t)}{E(0)}\mathcal{K}'\Bigg(\frac{\epsilon_1}{(t-t_0)^{\frac{1}{1+\epsilon_0}}}\frac{E(t)}{E(0)}\Bigg)
	\]
	by choosing $\epsilon_1$ such that $ mE(0)-c\epsilon_1 >0$, we obtain
	\begin{equation}\label{est_L3}
	L_3'(t) \leq -c\xi(t)\frac{E(t)}{E(0)}\mathcal{K}'\Bigg(\frac{\epsilon_1}{(t-t_0)^{\frac{1}{1+\epsilon_0}}}\frac{E(t)}{E(0)}\Bigg)
	\end{equation}
	Integrating (\ref{est_L3}) from $t_1$ to $t$ to get, 
	\[
	\displaystyle \int_{t_1}^t c\xi(\tau)\frac{E(\tau)}{E(0)}\mathcal{K}'\Bigg(\frac{\epsilon_1}{(\tau-t_0)^{\frac{1}{1+\epsilon_0}}}\frac{E(\tau)}{E(0)}\Bigg)d\tau \leq - \int_{t_1}^t L_3'(t)dt \leq L_3(t_1)
	\]
	Using the properties of $\mathcal{K}, E$ and $\forall t\geq t_1,$
	\begin{multline*}
	c\frac{E(t)}{E(0)}\mathcal{K}'\Bigg(\frac{\epsilon_1}{(t-t_0)^{\frac{1}{1+\epsilon_0}}}\frac{E(t)}{E(0)}\Bigg) \displaystyle\int_{t_1}^t \xi(\tau)d\tau \\
	\leq c\xi(\tau)\frac{E(\tau)}{E(0)}\mathcal{K}'\Bigg(\frac{\epsilon_1}{(\tau-t_0)^{\frac{1}{1+\epsilon_0}}}\frac{E(\tau)}{E(0)}\Bigg)d\tau \leq L_3(t_1).
	\end{multline*}
	Hence, 
	\[
	c\Bigg(\frac{1}{(t-t_0)^{\frac{1}{1+\epsilon_0}}}\frac{E(t)}{E(0)}\Bigg)\mathcal{K}'\Bigg(\frac{\epsilon_1}{(t-t_0)^{\frac{1}{1+\epsilon_0}}}\frac{E(t)}{E(0)}\Bigg)\int_{t_1}^t \xi(\tau)d\tau \leq \frac{c_1}{(t-t_0)^{\frac{1}{1+\epsilon_0}}}.
	\]
	Setting $\mathcal{K}_1(\tau)=\tau\mathcal{K}'(\epsilon_1\tau)$, the above inequality reduces to
	\begin{equation}
	\label{final_energy_estimate}
	c\mathcal{K}_1\Bigg(\frac{\epsilon_1}{(t-t_0)^{\frac{1}{1+\epsilon_0}}}\frac{E(t)}{E(0)}\Bigg)\int_{t_1}^t \xi(\tau)d\tau \leq \frac{c_1}{(t-t_0)^{\frac{1}{1+\epsilon_0}}},
	\end{equation}
	After rearranging the terms in \eqref{final_energy_estimate}, we conclude that 
	\[
	E(t) \leq c(t-t_0)^{\frac{1}{1+\epsilon_0}} \mathcal{K}_1^{-1}\Bigg( \frac{c_1}{(t-t_0)^{\frac{1}{1+\epsilon_0}}\int_{t_1}^t \xi(\tau)d\tau} \Bigg), \ \forall t\geq t_1.
	\]
	Hence, the theorem is proved.
\end{proof}

\begin{remark}
	If $H(0)=0$ and $H$ is strictly convex on $(0,r]$, then 
	\begin{equation}
	\label{convex_prop}
	H(\theta s) \leq \theta H(s),\ \theta\in [0,1],\ {\rm \ and\ } s\in (0,r].
	\end{equation}
\end{remark}

\begin{lemma}
	\label{lemma_est_barH}
	Assume that there exists $C$ such that $\tau h(\tau)\leq C$. Then for some constant $c>0$, we have the following estimate:
	\begin{equation}
	\label{est_barH}
	\bar{H}^{-1}(G(t)) \leq c(t-t_1)^{\frac{1}{1+\epsilon}} \bar{H}^{-1}\Bigg( \frac{M(t)}{(t-t_1)^{\frac{1}{1+\epsilon}}}\Bigg)^\frac{1}{1+\epsilon}.
	\end{equation}		
\end{lemma}
\begin{proof}
	Since, $\lim\limits_{t\rightarrow \infty} \frac{1}{t-t_1}=0,\ \exists t_2$ such that $\frac{1}{t-t_1}<1$ whenever $t>t_2$. Using the strict convexity and strictly increasing properties of $\bar{H}$, with $\theta = \frac{1}{(t-t_1)^{\frac{1}{1+\epsilon}}} <1$ and using (\ref{convex_prop}), we get
	\begin{equation}
	\label{est_HG1}
	\bar{H}^{-1}(G(t)) \leq (t-t_1)^{\frac{1}{1+\epsilon}} \bar{H}^{-1}\Bigg( \frac{M(t)}{(t-t_1)^{\frac{1}{1+\epsilon}}}\Bigg).
	\end{equation}
	Since, $\tau h(\tau)\leq C$ it is easy to see that $M(t)\leq C$ and also $\frac{M(t)}{(t-t_1)^{\frac{1}{1+\epsilon}}} \leq C$. Therefore 
	
	\begin{equation}
	\label{est_HG2}
	\begin{array}{ll}
	\bar{H}^{-1}\Bigg( \frac{M(t)}{(t-t_1)^{\frac{1}{1+\epsilon}}}\Bigg) & = \bar{H}^{-1}\Bigg( \frac{M(t)}{(t-t_1)^{\frac{1}{1+\epsilon}}}\Bigg)^\frac{\epsilon}{1+\epsilon} \bar{H}^{-1}\Bigg( \frac{M(t)}{(t-t_1)^{\frac{1}{1+\epsilon}}}\Bigg)^\frac{1}{1+\epsilon}\medskip \\
	&\leq c\bar{H}^{-1}\Bigg( \frac{M(t)}{(t-t_1)^{\frac{1}{1+\epsilon}}}\Bigg)^\frac{1}{1+\epsilon}
	\end {array} 
	\end{equation}
	Therefore (\ref{est_barH}) follows from \eqref{est_HG1} and \eqref{est_HG2}. Hence, the lemma is proved.
\end{proof}

\begin{theorem}
	Let $(u_0,u_1)\in H_0^2(\Omega)\times H_0^2(\Omega)$ and $h_1, B$ are nonlinear. Assume that the hypothesis of Lemma \ref{lemma_est_barH} and the hypothesis $(H1)-(H4)$ holds. Then for all $t\geq t_1$, we have 
	\begin{equation}
	\label{energy_est_nonlinear}
	E(t) \leq c \big( t-t_0 \big)^{\frac{1}{1+\epsilon}} W_2^{-1}\Big( \frac{c}{\big( t-t_0 \big)^{\frac{1}{1+\epsilon}} \int_{t_0}^t\xi(s)ds} \Big),
	\end{equation}
	where $W_2(t)=tW'(\epsilon_1t)$ and $W(t)=\Big( \big(\bar{B}^{-1}\big)^{\frac{1}{1+\epsilon}} + \big(\bar{H}^{-1}\big)^{\frac{1}{1+\epsilon}} \Big)^{-1}(t).$
\end{theorem}

\begin{proof} Using (\ref{est_L'}) and Lemma (\ref{est_h1}), observe that
	\begin{equation*}
	\begin{array}{ll}
	L'(t) &\leq  -mE(t) + c(b\circ \Delta u)(t) + c(b\circ \Delta u)^{\frac{1}{1+\epsilon_0}}(t) + c\int_{\Omega} h^2(u_t)dx\medskip \\
	& \leq -mE(t) + c(b\circ \Delta u)(t) + c(b\circ \Delta u)^{\frac{1}{1+\epsilon_0}}(t)+cH^{-1}(M(t))+c(-E'(t)).
	\end {array} 
	\end{equation*}
	Denote $L_1(t)=L(t)+cE(t)$, then using Remark \ref{rmk_bu}, Lemma's \ref{est_b_t} and \ref{lemma_est_barH} and for all $t\geq t_1$ the above inequality becomes
	\begin{equation*}
	\begin{array}{ll}
	L_1'(t) & \leq  -mE(t) + c(b\circ \Delta u)(t) + c(b\circ \Delta u)^{\frac{1}{1+\epsilon_0}}(t)+cH^{-1}(M(t))\medskip \\
	& \leq -mE(t) + c(b\circ \Delta u)^{\frac{1}{1+\epsilon_0}}(t)+cH^{-1}(M(t))\medskip \\
	& \leq -mE(t) + c(t-t_0)^{\frac{1}{1+\epsilon_0}} \bar{B}^{-1}\Big( \frac{\delta M(t)}{(t-t_0)\xi(t)} \Big)^{\frac{1}{1+\epsilon_0}}\medskip \\
	& \hspace{3cm}  +  c(t-t_0)^{\frac{1}{1+\epsilon}} \bar{H}^{-1}\Bigg( \frac{M(t)}{(t-t_0)^{\frac{1}{1+\epsilon}}}\Bigg)^\frac{1}{1+\epsilon}\medskip \\
	& \leq -mE(t) + c(t-t_0)^{\frac{1}{1+\epsilon_0}} \bar{B}^{-1}\Big( \frac{\delta M(t)}{(t-t_0)^{\frac{1}{1+\epsilon_0}}\xi(t)} \Big)^{\frac{1}{1+\epsilon_0}}\medskip \\
	& \hspace{3cm}  
	+  c(t-t_0)^{\frac{1}{1+\epsilon}} \bar{H}^{-1}\Bigg( \frac{M(t)}{(t-t_0)^{\frac{1}{1+\epsilon}}}\Bigg)^\frac{1}{1+\epsilon},
	\end {array} 
	\end{equation*}
	Denote
	\begin{equation}
	\label{notation_1}
	W(t)=\Big( \big(\bar{B}^{-1}\big)^{\frac{1}{1+\epsilon}} + \big(\bar{H}^{-1}\big)^{\frac{1}{1+\epsilon}} \Big)^{-1}(t),
	\end{equation}
	and
	\begin{equation}
	\label{notation_2}
	\beta(t)=\max \Biggl\{ \frac{\delta M(t)}{(t-t_0)^{\frac{1}{1+\epsilon_0}}\xi(t)}, \frac{M(t)}{(t-t_1)^{\frac{1}{1+\epsilon}}} \Biggr\}.
	\end{equation}
	From the definition of $W$, it is clear that $W'(t)>0$ and $W''(t)>0$ on $(0,r]$, where $r=\min\{r_1,r_2\}$. Using these notations, we obtain for all $t\geq t_1$,
	\begin{equation}
	L_1'(t) \leq  -mE(t) + c(t-t_0)^{\frac{1}{1+\epsilon_0}}W^{-1}\big(\beta(t)\big).
	\end{equation}
	To conclude (\ref{energy_est_nonlinear}), we follow the similar lines as shown in the proof of Theorem \ref{energy_est_linear_h}. This completes the proof of this theorem.  
\end{proof}

\bibliography{existence}

		\end{document}